\documentclass[11pt]{article}
\usepackage[utf8]{inputenc}
\usepackage{amsmath,amsthm}
\usepackage{amsfonts}
\usepackage{amssymb}
\usepackage{array}
\usepackage{bm}
\usepackage{multirow}
\usepackage{subfig}
\usepackage[shortlabels]{enumitem}
\usepackage{microtype}
\usepackage{xspace}
\usepackage[normalem]{ulem}
\usepackage[font=small]{caption}

\usepackage[margin=2.9cm]{geometry}
\RequirePackage[colorlinks,citecolor=blue,urlcolor=blue]{hyperref}

\usepackage{algorithm}
\usepackage{algpseudocode}

\algnewcommand{\IIf}[1]{\State\algorithmicif\ #1\ \algorithmicthen}

\usepackage{color}
\usepackage{enumitem}

\usepackage[style=alphabetic,
            backend=bibtex,
            doi=false,
            isbn=false,
            url=false,
            giveninits=true,
            maxnames=4
            ]{biblatex}

\makeatletter
\def\blx@maxline{77}
\makeatother

\renewbibmacro{in:}{}           
\DeclareFieldFormat[article,inbook,incollection,inproceedings,patent,thesis,unpublished]{citetitle}{#1}
\DeclareFieldFormat[article,inbook,incollection,inproceedings,patent,thesis,unpublished]{title}{#1} 

\usepackage{arash_macros}
\usepackage{authblk}
\title{On perfectness in Gaussian graphical models}
\author{Arash A. Amini, Bryon Aragam, Qing Zhou}

\providecommand\normprc{$\mathsf{normPrc}$\xspace}
\providecommand\normcov{$\mathsf{normCov}$\xspace}

\providecommand{\Lc}{\mathcal{L}}

\def\sep(#1,#2,#3){#1-#2-#3}

\providecommand\independent{\protect\mathpalette{\protect\independenT}{\perp}}
\def\independenT#1#2{\mathrel{\rlap{$#1#2$}\mkern4mu{#1#2}}}

\def\given{\,\mid\,}

\definecolor{maroon}{rgb}{0.5, 0.0, 0.0}

\providecommand{\complex}{\mathbb{C}}
\providecommand{\Cyc}{\mathcal{S}}
\providecommand{\Cc}{\mathcal{C}}
\newcommand{\Pc}{\mathcal{P}}
\newcommand{\Dc}{\mathcal D}

\providecommand{\rv}{X}

\newcommand{\sphere}{\mathbb S}

\newcommand{\Rc}{\mathcal{R}}
\newcommand{\dual}{\rceil}
\newcommand{\ipp}[1]{\langle \!\langle #1 \rangle\! \rangle}

\newcommand{\Mc}{\mathcal{M}}
\newcommand\Hc{\mathcal{H}}

\newcommand{\Nc}{\mathcal N}
\newcommand{\Np}{\mathcal N}

\begin{document}
\maketitle

\begin{abstract}
Knowing when a graphical model is perfect to a distribution is essential  in order to relate separation in the graph to conditional independence in the distribution, and this is particularly important when performing inference from data. When the model is perfect, there is a one-to-one correspondence between conditional independence statements in the distribution and separation statements in the graph. Previous work has shown that  almost all models based on linear directed acyclic graphs as well as Gaussian chain graphs are perfect, the latter of which subsumes Gaussian graphical models (i.e., the undirected Gaussian models) as a special case. However, the complexity of chain graph models leads to a proof of this result which is indirect and mired by the complications of parameterizing this general class. 
In this paper, we directly approach the problem of perfectness for the Gaussian graphical models, and provide a new proof, via a more transparent parametrization, that almost all such models are perfect. Our approach is based on, and substantially extends, a construction of Ln{\v{e}}ni{\v{c}}ka and Mat{\'{u}}{\v{s}} showing the existence of a perfect Gaussian distribution for any graph. 
\end{abstract}

\section{Introduction}

Graphical models are among the most common approaches to modeling dependencies in multivariate data~\cite{Lauritzen1996,koller2009}. To be concrete, consider a random vector $X = (X_1,\dots,X_d) \in \reals^d$. The general idea behind graphical modeling is to represent the conditional independence (CI) statements satisfied by the multivariate distribution $\pr$ of $X$ by the \emph{separating sets} in a graph $G=(V,E)$ with nodes $V=\{X_1,\dots,X_d\}$. Whenever graph separation in $G$ implies conditional independence in $\pr$, the distribution is said to be \emph{Markovian} 
with respect to (w.r.t.) $G$ and we have a graphical model for $\pr$ (see Section~\ref{subsec:gm} for details). In this paper, we focus on undirected graphs (UGs), in which case, $G$ is called a conditional independence graph (CIG) for $\pr$.

A question that arises is to what extent such correspondence is possible for a given distribution. A particular case of interest is when the correspondence is exact, that is, the set of CI statements entailed by the distribution is the same as the set of separation statements in the graph. If this desirable property holds, the distribution $\pr$  is said to be \emph{perfect} with respect to the graph $G$.  
In other words, in the perfect case, both the Markov property above and its reverse implication hold (i.e., CI in $\pr$ implies graph separation in $G$ as well). Thus, we can ``read off'' the CI relations in $\pr$ by inspecting  the graph $G$. Moreover, the graph $G$ provides an economical representation of these relations that can be learned from data \cite{koller2009,spirtes2000}.

In previous work \cite{spirtes1993causation,meek1995strong,levitz2001separation,pena2011}, it has been shown that almost all linear directed acyclic graph (DAG) and Gaussian chain graph (CG) models are perfect.
In this work, we consider the case of undirected Gaussian graphical models (GGMs),
i.e. $X \sim N(0,\Sigma)$, and show that almost all of them are perfect. In other words, almost all Gaussian distributions are capable of being perfectly represented by an undirected graph $G$. 
Technically speaking, the results of~\cite{levitz2001separation,pena2011}  already show perfectness for almost all Gaussian distributions that factor according to a UG (i.e. as a special case of a CG), however, the constructions and proofs are obscured by the complexity of the CG case. In particular, although showing essentially the same result, \cite{pena2011} and~\cite{levitz2001separation} use two different indirect parametrizations of the CG-Markovian Gaussian distributions. 
In this paper, we provide a much simpler and more direct parameterization for the undirected case, which should be of independent interest. Our technique is based on an elegant construction of~\cite{Lnenicka2007} which was used to prove the existence of a perfect Gaussian distribution for any given UG. We extend this construction to a full parametrization of the UG-Markovian Gaussian distributions and prove the so-called \emph{strong completeness} of this class (i.e. that almost all are perfect). 

Our proof contains a simpler constructive description of the set of imperfect covariance matrices, which provides useful intuition for understanding perfectness assumption in modeling and estimation with UGs. 
As a byproduct of our proof, we construct a probability measure over inverse covariance matrices supported on the edge set of a graph $G$. This measure may be used as a trial or proposal distribution in Monte Carlo algorithms to simulate from many distributions over positive definite matrices with support restriction.

  The paper is organized as follows: Section~\ref{subsec:related} reviews related work and Section~\ref{sec:Gauss:graph:mod} provides some background on graphical modeling and sets up the notation. Section~\ref{sec:main:res} contains the statement of the main result and some discussion. Section~\ref{sec:construction} provides the details of our parameterization of Markovian distributions, the construction of the null set of perfect distributions and a more technical version of our main result. The proof of the main result appears in Section~\ref{sec:proofs} with the proof of some of the technical lemmas deferred to Section~\ref{sec:proof:aux}. 

\subsection{Related work}
\label{subsec:related}

The notion of perfect graphical models has a long history, and we refer the reader to textbooks such as \cite{pearl1988,koller2009} for details. For example, the problem of testing whether or not a given graph is perfect for a distribution has been studied in recent works \cite{tatikonda2014testing,sadeghi2017faithfulness,soh2018identifiability}.
In this paper, we focus on a related but distinct question: \emph{Given a graph $G$, how likely is it that a random Gaussian distribution is perfect with respect to $G$?} Making this statement precise requires a bit of care; see Section~\ref{sec:main:res}. Similar results are already known for other classes of graphical models. For DAGs, Markov perfectness, also known as \emph{faithfulness}, was shown in~\cite{spirtes1993causation,meek1995strong}. Using the same techniques, the result was extended to Gaussian distributions that factor according to chain graphs in~\cite{levitz2001separation,pena2011}. Chain graphs allow for both directed and undirected edges and the corresponding graphical models extend both the UG and DAG models. There are two equivalent formulations of the Markov property for chain graphs referred to as the Andersson--Madigan--Perlman (AMP) versus the Lauritzen--Wermuth--Frydenberg (LWF) interpretation~\cite{Lauritzen1996,andersson2001alternative,studeny2006probabilistic}. 
In~\cite[Section~6]{levitz2001separation}, perfectness of almost all Gaussian distributions that are Markovian w.r.t. to a CG was shown using the AMP interpretation. A similar result was obtained in~\cite{pena2011} using the LWF interpretation.

\subsection{Gaussian graphical models}\label{sec:Gauss:graph:mod}
\label{subsec:gm}

Consider an undirected graph $G=([d],E)$, where $[d] = \{1,\dots,d\}$. 
Two nodes $i$ and $j$ are \emph{adjacent}, or \emph{neighbors}, if $(i,j)\in E$, in which case we write $i \sim j$, otherwise $i \nsim j$. A \emph{path} from $i$ to $j$ is a sequence $i = k_1,k_2,\ldots,k_{n-1},k_{n}=j \in[d]$ of distinct elements with $(k_{\ell},k_{\ell+1})\in E$ for each $\ell=1,\ldots,n-1$. 
Given two subsets $A,B\subset [d]$, a path connecting $A$ to $B$ is any path with $k_{1} \in A$ and $k_{n}\in B$.
A subset $C\subset [d]$ \emph{separates} $A$ from $B$, denoted by  $A-C-B$, if all paths connecting $A$ to $B$ intersect $C$ (i.e. $k_{\ell}\in C$ for some $1<\ell<n$), otherwise we write $ \neg (\sep(A,C,B))$. Implicit in this definition is that  $A,B$ and $C$ are disjoint.

To simplify the notation, we often identify $G$ with its edge set $E$, i.e., $G \simeq E$. For example, we also write $|G| := |E|$ to denote the number of edges. We also adopt the following shorthands: $\{i\} = i$ and $\{i,j\} = ij$, $A \cup \{i\} = Ai$, $A \cup B = AB$ and so on, that is, the union of sets is denoted by juxtaposition.
In addition, we let $[d]_S = [d] \setminus S = \{1,\dots,d\} \setminus S$. Common uses of these notational conventions are: $[d]_j = [d] \setminus \{j\}$ and $[d]_{ij} = [d] \setminus ij = [d] \setminus \{i,j\}$.
For a matrix $\Sigma \in \reals^{d \times d}$, and subsets $A,B \subset [d]$, we use $\Sigma_{A,B}$ for the submatrix on rows and columns indexed by $A$ and $B$, respectively. Single index notation is used for principal submatrices, so that $\Sigma_{A} = \Sigma_{A,\,A}$. For example, $\Sigma_{i,j}$ is the $(i,j)$th element of $\Sigma$ (using the singleton notation), whereas $\Sigma_{ij} = \Sigma_{ij,\,ij}$ is the $2\times 2$ submatrix on $\{i,j\}$ and $\{i,j\}$. 
Similarly, $\Sigma_{Ai,Bj}$ is the submatrix indexed by rows $A \cup \{i\}$ and columns $B \cup \{j\}$.

Now, consider a random vector $X = (X_1,\dots,X_d) \in \reals^d$ and a graph $G$ on nodes $[d]$ where node $i$ represents random variable $X_i$.  A random vector $X$ (or its distribution $\pr$) is called \emph{Markovian} w.r.t. $G$ (and $G$ a CIG for $X$) if  
\begin{align}\label{eq:markovian:def}
	\text{$\sep(A,C,B)$ \;in $G \implies$ $\rv_{A}\independent \rv_{B}\given \rv_{C}$} \;\text{in $\pr$}.
\end{align}
Here, $\rv_{S} = \{\rv_i:\; i \in S\}$ for any $S\subset[d]$.
That is, the separation of the nodes in $A$ and $B$ by the nodes in $C$ implies that $\rv_A$ is independent of $\rv_B$ given $\rv_C$. The special case where~\eqref{eq:markovian:def} is assumed to hold only for sets of the form $A = \{i\}$, $B=\{j\}$ and $C = [d] \setminus \{i,j\}$ is called the \emph{pairwise Markov property}. This special case implies the full condition~\eqref{eq:markovian:def} if the distribution has a positive and continuous density w.r.t. a product measure on $\reals^d$~\cite[p.~34]{Lauritzen1996}.

Even if~\eqref{eq:markovian:def} holds, the converse need not necessarily hold. When the reverse implication of \eqref{eq:markovian:def} is true, we say the distribution of $\rv$ is perfect with respect to  graph $G$, or simply $G$ is perfect for $X$:

\begin{defn}
	A graph $G$ is \emph{perfect} for $X$ if $\sep(A,C,B)$ in $G \iff \rv_{A}\independent \rv_{B}\given  \rv_{C}$ in~$\pr$.
\end{defn}

In the Gaussian case, we have $X \sim N(0,\Sigma)$ where $\Sigma = (\Sigma_{i,j}) \in \reals^{d\times d}$ is the covariance matrix of $X$, that is, $\Sigma_{i,j} = \ex [X_i X_j]$. Using known results on Gaussian pairwise conditional independence \cite[Prop.~5.2]{Lauritzen1996}, $X_i \independent X_j \mid X_{[d]_{ij}}$ if and only if $[\Sigma^{-1}]_{i,j} = 0$. Thus, letting $G$ be defined by
\begin{align}\label{eq:inv:covariance}
\text{$i \nsim j$ in $G$} \iff [\Sigma^{-1}]_{i,j} = 0,
\end{align}
for $i \neq j$,
we have that $X$ (or $N(0,\Sigma)$ or $\Sigma$) satisfies the pairwise Markov property w.r.t.~$G$. Assuming that $\Sigma \succ 0$, it follows that $X$ w.r.t.~$G$ satisfies the (global) Markov property, hence $G$ is a CIG for $X$. Throughout, we will make the assumption $\Sigma \succ 0$, or equivalently that the Gaussian distribution is \emph{regular}.

From the above discussion, in the Gaussian case, Markov properties and CIGs can be equivalently characterized by the covariance matrix $\Sigma$.
 Thus, we can equivalently talk about \emph{perfectness of a covariance matrix}. The corresponding graph is uniquely implied in this case, given by the support of $\Sigma^{-1}$, i.e., $\supp(\Sigma^{-1}) := \{ (i,j) : (\Sigma^{-1})_{i,j} \neq 0, \; i < j\}$.
 We caution the reader that while the graph $G$ has $|G|$ edges by definition, the support of $\Sigma^{-1}$ has $|G|+d$ elements.  We will write $G^\circ$ for the graph $G$ with self-loops added, i.e., edges of the form $(i,i)$ for all $i \in [d]$. Then we have $|G^{\circ}|=|\supp(\Sigma^{-1})|=|G|+d$. The above discussion is summarized in the following definition:
 
 \begin{defn}
 	A positive definite matrix $\Sigma$ is  Markovian w.r.t. graph $G$ if $\supp(\Sigma^{-1}) = G^\circ$. It is perfect w.r.t. $G$ if $N(0,\Sigma)$  is so.
 \end{defn}

\section{Main result}\label{sec:main:res}

 In~\cite{Lnenicka2007}, it was~shown that for any graph $G$, there exists a regular Gaussian distribution which is  perfect w.r.t.~$G$.  As discussed in Section~\ref{sec:Gauss:graph:mod}, given any positive definite matrix $\Sigma$, we can ask whether it is perfect or not, with the graph of $G$ being implicit from the support of $\Sigma^{-1}$. This is the language that we will use throughout.
 The result of~\cite{Lnenicka2007} can be restated as follows: for any potential CIG, there is at least one covariance matrix $\Sigma$ which is perfect w.r.t. it.
Here, we extend the argument in~\cite{Lnenicka2007} to show that \emph{almost all} covariance matrices are perfect.

 \begin{thm}\label{thm:perfect:perval}
 	For any undirected graph $G$ on $[d]$, the set of positive definite matrices $A \in \reals^{d \times d}$ for which  $\Sigma = A^{-1}$ is Markovian but not perfect w.r.t. $G$ has Lebesgue measure zero. 
 \end{thm}
\noindent
 In Theorem~\ref{thm:perfect:perval} (and its corollary below), the Lebesgue measure is of dimension $|G^\circ| = |G| + d$.

 According to the discussion in Section~\ref{sec:Gauss:graph:mod}, $\Sigma$ is Markovian w.r.t. $G$ if $\Sigma^{-1}$ is supported on $G^\circ$.
 It follows that the set of matrices $A \in \reals^{d \times d}$ for which  $\Sigma = A^{-1}$ is Markovian w.r.t. $G$ can be identified with a set in $\reals^{|G^\circ|}$ of positive Lebesgue measure.
Theorem~\ref{thm:perfect:perval} then states that those $A$ in this set whose inverse is not perfect occupy a null subset. An equivalent restatement of the the result in terms of probability distributions with Markovian $\Sigma$ is the following:
 
\begin{cor}\label{thm:perfect:perval:2}
	Let $G$ be an undirected graph on $[d]$, and let $A \in \reals^{d \times d}$ be drawn from a  continuous distribution (w.r.t. the Lebesgue measure) on positive definite matrices with support  $G^\circ$.
	Then with probability one, $\Sigma = A^{-1}$ is perfect w.r.t to $G$.
\end{cor}
Theorem~\ref{thm:perfect:perval} is a consequence of a more technical result, Theorem~\ref{thm:zero:measure:main}, which is discussed in Section~\ref{sec:proof:perfect:perval} and could be of independent interest. One needs a fair amount of technical work to make the notion of ``almost all'' precise. This is done in Section~\ref{sec:proof:perfect:perval} by constructing appropriate measures on a suitable parametrization of the set of covariance matrices that are Markovian  w.r.t. a given graph $G$.
Once done, the same techniques in~\cite{Lnenicka2007} can be extended to show the stronger result as illustrated in the proof of Theorem~\ref{thm:zero:measure:main}. In addition, Theorem~\ref{thm:perfect:perval} further strengthens this result by showing that the notion of ``almost all'' is independent of the particular parametrization of Theorem~\ref{thm:zero:measure:main}.

\section{Construction of null sets}
\label{sec:construction}

 We begin by discussing how to parametrize the space of Markovian covariance matrices. We then characterize the subset of perfect covariance matrices in Theorem~\ref{thm:zero:measure:main}, a result that is interesting in its own right.

\subsection{A parametrization of Markovian covariance matrices}
\label{sec:proof:perfect:perval}

We now give a parametrization of the Markovian covariance matrices that provide a simple way of putting distributions on them. It also allows us to explicitly construct the set of imperfect covariance matrices from pieces that are all Lebesgue null sets.

Let $\sym^d$ be the set of symmetric $d \times d$ matrices,  $\sym_{++}^d$ the set of $d \times d$ positive definite matrices, and define 
 \[\sym_{++,1}^d = \{\Gamma\in \sym_{++}^d:\; \Gamma_{i,i} = 1,\; i \in [d] \}. \]
  Matrices in $\sym_{++,1}^d$ are often called correlation matrices. Since we use this normalization mainly for precision matrices, to avoid confusion, we  call elements of $\sym_{++,1}^d $ \normprc matrices. It is not hard to see that 
for any diagonal matrix $D\in \sym_{++}^d$, the two matrices $D \Sigma D$ and $\Sigma$ have the same Markov properties.  Thus, it is enough to focus on the case where $(\Sigma^{-1})_{i,i} = 1$ for all $i \in [d]$. We will make the following shorthand:
\begin{defn}
	A matrix $\Sigma$ is called a \normcov matrix if $\Sigma^{-1} \in  \sym_{++,1}^d$, i.e., its inverse is a normalized precision matrix.
\end{defn}

Given any graph $G$ on nodes $[d]$, our first step is to construct a uniform measure over all \normcov matrices that are Markovian w.r.t. $G$.  We then show that with probability one, such \normcov matrices are perfect. Later, we will show how the result extends to all covariance matrices Markovian w.r.t. $G$ (see Step~2 in the proof of Theorem~\ref{thm:perfect:perval}).
The class of \normcov matrices that are Markovian w.r.t. $G$ can be written as
\begin{align*}
\Psi_G^{-1} := \{\Gamma^{-1}:\; \Gamma \in \Psi_G\}, \quad \text{where} \quad \Psi_{G}:= \{\Gamma \in \sym_{++,1}^d:\; \Gamma_{i,j} \neq 0 \iff ij \in G\}.
\end{align*}
$\Psi_G$ is just the set of \normprc matrices with support  $G$. The first step in our approach is to put a distribution on $\Psi_G^{-1}$ as the push-forward of a distribution constructed on $\Psi_G$. Although our construction is not uniform w.r.t. the Lebesgue measure, in Corollary~\ref{thm:perfect:perval:2} we extend the result to any distribution on $\Psi_G$ which is absolutely continuous w.r.t. the Lebesgue measure.

Before describing our construction of a random \normcov matrix, let us set up some more notation. We let $\Lc^n$ and $\Hc^s$  ($s > 0$) denote, respectively, the Lebesgue measure and the $s$-dimensional Hausdorff measure on $\reals^n$. The dimension of the ambient space of $\Hc^s$ will be clear from the context.  For the graph $G$, let $g = |G|$ be the number of edges.  We often identify $\Psi_G$ with a subset of $\reals^G$, and often identify $\reals^G$ 
with $\reals^g$, after ordering the edges, the particular order being unimportant. For example, if $G$ is  $1 - 2 -3$, with $g = |G| = 2$, and $\Gamma \in \Psi_G$ is
\begin{align*}
\Gamma = \begin{pmatrix}
1 & \delta_{12} & 0 \\
\delta_{12} & 1 & \delta_{23} \\
0 & \delta_{23} & 1\\
\end{pmatrix},
\end{align*}
we either view $\Gamma$ as $\{\delta_{12},\delta_{23}\} = (\delta_{ij}, ij \in G)$, as an element of $\reals^G$, or as the ordered pair $(\delta_{12},\delta_{23})$  as an element of $\reals^g = \reals^2$.

For $\delta = (\delta_{ij},\; ij \in G) \in \reals^{G}$ and $\eps > 0 $, define $A^{G,\delta,\eps} = (a_{ij}^{G,\delta,\eps}) \in \reals^{d \times d}$  by setting 
\begin{align}\label{eq:A:G:delta:eps}
a_{ij}^{G,\delta,\eps} = \begin{cases}
\delta_{ij} \,\eps\, & ij \in G \\
1 & i=j \\
0 & \text{otherwise} \end{cases}.
\end{align}
For a fixed $\delta \in \reals^{G}$, let $\eps_G(\delta)$ be the largest $\eps > 0$ such that $A^{G,\delta,\eps}$ is positive definite, that is,
\begin{align}
\label{eq:def:eps}
\eps_G(\delta) := \sup \{\eps > 0:\; A^{G,\delta,\eps} \in \sym_{++}^d\}. 
\end{align}
Then $A^{G,\delta,\eps}$ is positive definite for all $\eps\in[0,\eps_G(\delta))$, due to the convexity of $\sym_{++}^d$.
Let $[-1,1]_* :=[-1,1]\setminus \{0\}$ and consider 
\begin{align*}
\Mc^G := \big\{\, (\delta,\eps): \;\;\delta \in [-1,1]_*^{G}, \; \;\eps \in (0,\eps_G(\delta)) \, \big\}.
\end{align*}
The set $\{(A^{G,\delta,\eps})^{-1}: \; (\delta,\eps) \in \Mc^G\}$ is a parametrization of the set of all \normcov  matrices that are Markovian w.r.t. $G$. In other words, with the map $\zeta:\Mc^G\to\reals^{d \times d}$ \begin{align}\label{eq:zeta:def}
\zeta(\delta,\eps) = (A^{G,\delta,\eps})^{-1},
\end{align}
we have $\zeta(\Mc^G) = \Psi_G^{-1}$.
We note that $\Mc^G$ is a subset of $[-1,1]^G \times (0,\infty) \subset \reals^G \times \reals \simeq \reals^{g+1}$. We will equip $\Mc^G$ with the Lebesgue measure (i.e., $\Lc^{g+1}$).

The map $\zeta$ overparametrizes the set $\Psi_G^{-1}$ since $\zeta(c\delta,\eps/c)$ is the same for all $c>0$, i.e. it defines the same \normcov matrix for all $c>0$.
To remove this ambiguity (and to avoid unnecessary complications in working with equivalence classes), without loss of generality, we focus on the subset of $\Mc^G$ for which $\delta$ has unit $\ell_\infty$ norm. Let $ \sphere^G_{\infty} := \{\delta \in \reals^G:\; \infnorm{\delta} = 1\}$, $\sphere^G_{\infty,*} = [-1,1]_*^{G}\cap \sphere^G_\infty$, and
\begin{align*}
\Mc_\infty^G := \Mc^G \cap (\sphere_\infty^G \times \reals) = \big\{\, (\delta,\eps): \;\delta \in \sphere^G_{\infty,*},  \;\eps \in (0,\eps_G(\delta)) \, \big\}.
\end{align*}
The function $\eps_G$, restricted to $\sphere^G_{\infty,*}$, is continuous and bounded. In fact, $\sup \eps_G(\sphere^G_{\infty,*}) = 1$ so that $\Mc_\infty^G \subset [-1,1]^G \times (0,1]$. Hence $\Mc_\infty^G$ has finite and positive $\Hc^{g}$-measure (on $\reals^{g+1}$), where $g := |G|$. 
\begin{defn}\label{def:muG}
Equip $\Mc_\infty^G$ with the normalized  $\Hc^{g}$-measure, which defines a uniform probability distribution $\mu_G$ on  $\Mc_\infty^G$. It is equivalently described as follows:  Pick $\delta'$ by drawing each entry uniformly from $[-1,1]_*$, and given $\delta'$, set $\delta = \delta' / \infnorm{\delta'}$ and  draw $\eps$ uniformly from $[0,\eps_G(\delta)]$; the vector $(\delta,\eps)$ has the desired distribution. See Figure~\ref{fig:mu_G}.
\end{defn}

\begin{figure}
	\centering 
\includegraphics[height=1.75in]{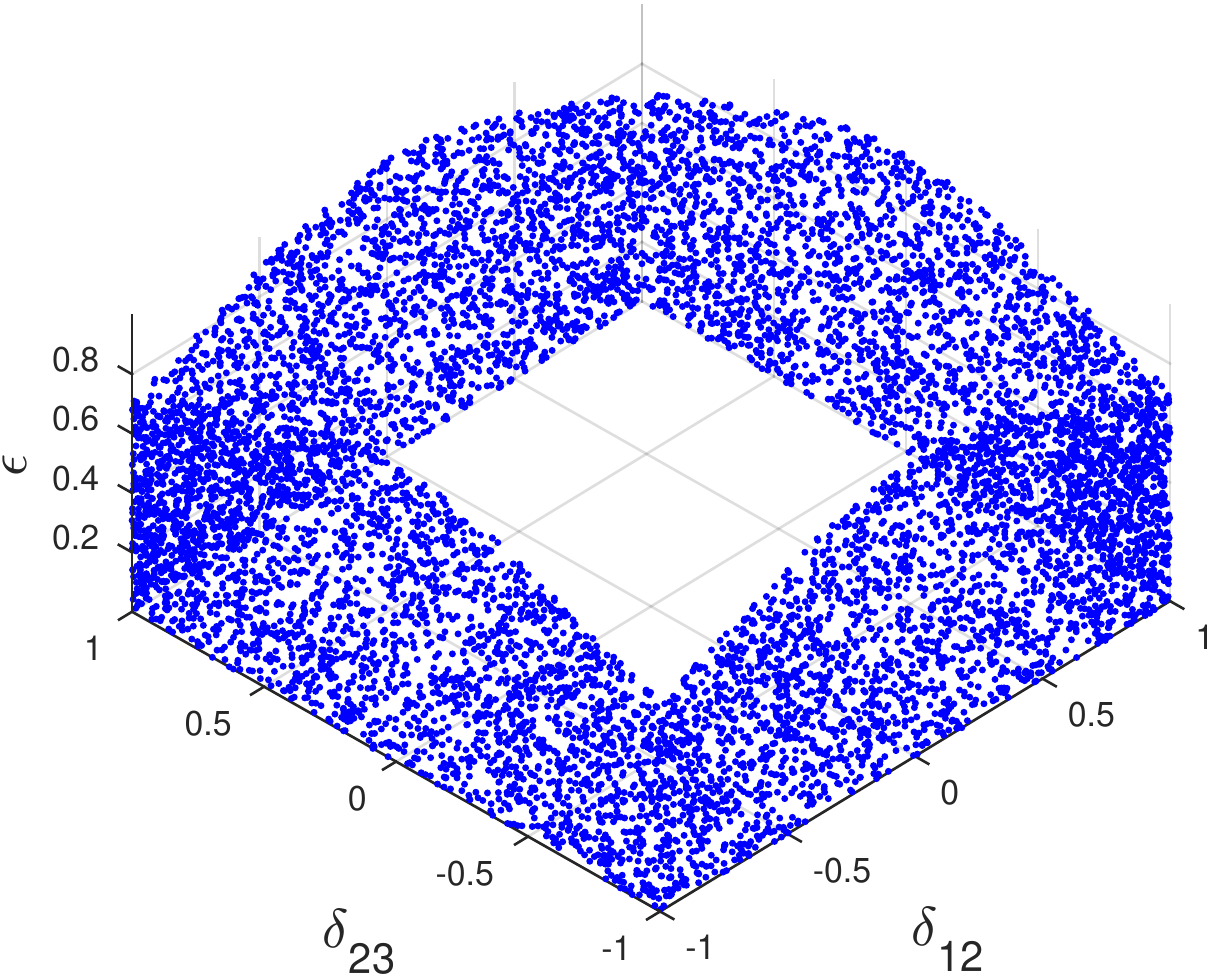}\qquad
	\includegraphics[height=1.75in]{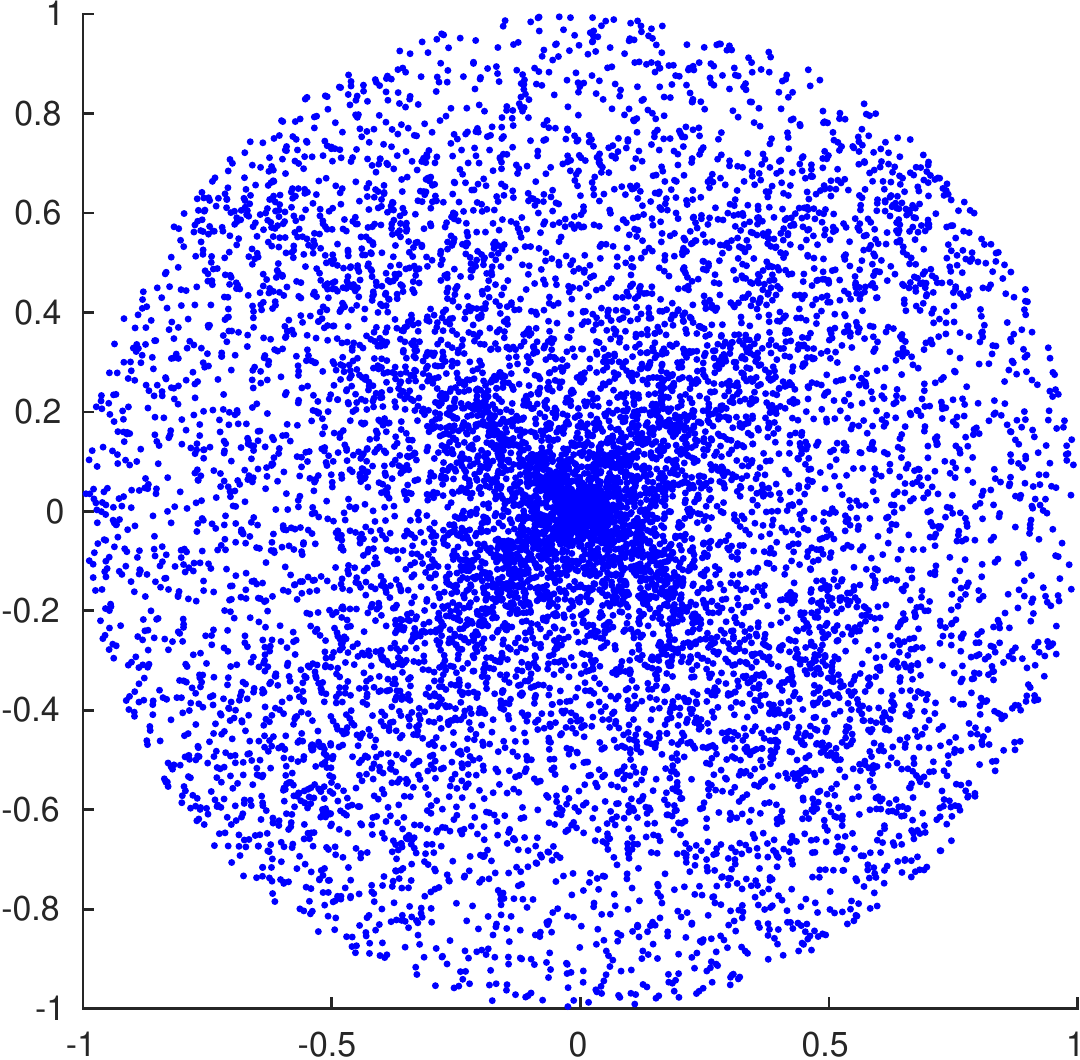}
	\caption{(Left) The plot of the samples from the uniform distribution  $\mu_G$  on $\Mc_\infty^G$ for the graph $G = 1 - 2 - 3$. We have $\eps_G((\delta_{12},\pm 1)) = 1/\sqrt{1+\delta_{12}^2}$ and similarly for $\eps_G((\pm 1, \delta_{23}))$. Note that the distribution is supported on $\sphere_\infty^2 \times [0,1]$. It is singular w.r.t. $\Lc^3$ but absolutely continuous w.r.t. $\Hc^2$. (Right) Pushforward of $\mu_G$ by the map $F$ defined in Step~1 in the proof of Theorem~\ref{thm:perfect:perval}. (cf.~Section~\ref{sec:proof:thm:perfect:preval})}
	\label{fig:mu_G}
\end{figure}
Restricted to $\Mc_\infty^G$, the map $\zeta$ defined earlier is well-behaved: It is one-to-one and onto
$\Psi_G^{-1}$, that is, $\zeta: \Mc_\infty^G \to \Psi_G^{-1}$ is a  bijection.
We can now put a distribution on \normcov matrices, $\Psi_G^{-1}$, as the push-forward of $\mu_G$ by $\zeta$.

\subsection{Characterizing imperfect covariance matrices}

Let us now consider the subclass of $\Psi_G^{-1}$ which is imperfect. It is enough to work with the corresponding subsets in $\Mc^G$ and $\Mc^G_\infty$:
\begin{align*}
\Np^G &= \{(\delta,\eps) \in \Mc^G:\; (A^{G,\delta,\eps})^{-1} \; \text{is not perfect} \},\\
\Np^G_\infty &=  \{(\delta,\eps) \in \Mc_\infty^G:\; (A^{G,\delta,\eps})^{-1} \; \text{is not perfect} \} = \Nc^G \cap (\sphere_\infty^G \times \reals).
\end{align*}
The following result is the key component on which Theorem~\ref{thm:perfect:perval} is based: \begin{thm}\label{thm:zero:measure:main}
	For $\delta \in [-1,1]_*^G$, let 
$B_\delta:= \{\eps:\; (\delta, \eps) \in \Nc^G \}$,
and let $g := |G| \ge 2$. 
There exists a set $\Dc \subset [-1,1]_*^G$ such that
	\begin{itemize}
		\item[(a)] $\Dc^c$ is an $\Lc^g$-null set, 
		\item[(b)] $\sphere^G_\infty \cap \Dc^c$  is an $\Hc^{g-1}$-null set, and
		\item[(c)] for every $\delta \in \Dc$, $B_\delta$ is finite. \end{itemize}
	In particular, (d) the set $\Nc^G$ is a $\Lc^{g+1}$-null set,
and $\Nc^{G}_\infty$ is a $\Hc^g$-null set, i.e. $\mu_G(\Nc^{G}_\infty)=0$.
\end{thm}

Theorem~\ref{thm:zero:measure:main} is proved in Section~\ref{sec:proof:zero:measure}. To gain some intuition for this result, consider the example $G=1-2-3$ illustrated in Figure~\ref{fig:mu_G}. Theorem~\ref{thm:zero:measure:main} says that there is a ``good'' set $\Dc$ of $(\delta_{12},\delta_{23}) \in [-1,1]_*^2$ which has full 2-dimensional measure; its boundary (i.e., the intersection with the perimeter of the square $[-1,1]_*^2$) has full 1-dimensional measure. Moreover, for any $(\delta_{12},\delta_{23}) \in \Dc$, at most finitely many $\eps$ are problematic, that is, they lead to a imperfect covariance matrix via~\eqref{eq:zeta:def}.

As part of the proof of Theorem~\ref{thm:zero:measure:main}, we give an explicit construction of the set $\Dc^c$ as an algebraic variety in $\reals^g$. See~\eqref{eq:strong:path:condition} in Section~\ref{sec:proof:zero:measure}. In addition,
  the measure $\mu_G$ (cf.~Definition~\ref{def:muG}) is itself interesting. 
Using $\mu_G$ as a trial or proposal distribution in rejection sampling or Markov chain Monte Carlo, it gives a device, at least in principle, to generate a random sample from any continuous distribution  over (properly normalized) covariance matrices that are Markovian w.r.t. any given graph $G$. The only computational burden in simulating from $\mu_G$ is computing $\eps_G(\delta)$, which can be obtained by solving a convex optimization problem.

\section{Proof of main results} 
\label{sec:proofs}
In this section, we provide the proofs of the two main theorems. We first show how Theorem~\ref{thm:zero:measure:main} implies Theorem~\ref{thm:perfect:perval}. The rest of the section is then focused on proving Theorem~\ref{thm:zero:measure:main}. The main component is Lemma~\ref{lem:cycle} whose proof is deferred to Section~\ref{sec:proof:cycle}. 
\subsection{Proof of Theorem~\ref{thm:perfect:perval}}\label{sec:proof:thm:perfect:preval}

	Recall the  identification of $\Psi_G$ with a subset of $\reals^G \simeq \reals^g$. The cases $g=0$ and $g=1$ are trivial, so we will assume $g \ge 2$. We proceed in two steps:
	
	\medskip
	\noindent \textbf{Step 1.} 
    The first step is to show that almost all \normprc matrices supported on $G^\circ$ lead to perfect covariance matrices after inversion.
	Consider the following subset of \normprc matrices, \begin{align}
	N = \{\Gamma \in \Psi_G:\; \text{$\Gamma^{-1}$ is not perfect}\}.
	\end{align}
We wish to show that $\Lc^g(N)=\Hc^g(N)=0$.
Let $F:\reals^G \times \reals \to \reals^G$ be the map given by $F(\delta,\eps) = (\delta_{ij} \eps,\, ij \in G)$. Since $F$ is a $C^1$ map, it is locally Lipschitz over $\reals^{|G|+1}$, hence  Lipschitz over $\Mc_\infty^G$.
It is well-known that a Lipschitz map (between two metric spaces) maps $\Hc^s$-null sets to $\Hc^s$-null sets,
for any $s > 0$; see for example~\cite[Proposition 2.4.7]{krantz2008geometric}.
Since $\Hc^g(\Nc_\infty^G) = 0$ according to Theorem~\ref{thm:zero:measure:main}, it follows that $\Hc^{g}( F (\Nc_\infty^G)) = 0$. 
Recalling the identification of  $\Psi_G$ with a subset of $\reals^G$, and that $F: \Mc^G_\infty \to \Psi_G$ is a bijection, we have $F (\Nc_\infty^G) \simeq N$, that is, $\Hc^g(N) = 0$.
	
	\medskip    
	\noindent \textbf{Step 2.} 
    The second step is to extend the previous result for \normprc matrices to general positive definite matrices. Let $\reals_{++}$ be the set of positive reals
and let $\xi: \reals_{++}^d \times \reals^{G} \to \reals^G$ be given by 
	\begin{align}\label{eq:rescaling:map:defn}
	\xi : \;\big( x_k, k \in [d];\; y_{ij}, \; ij \in G \big) \mapsto 
	\Big(\frac{y_{ij}}{\sqrt{x_i x_j}},\, ij \in G \Big).
	\end{align}
	This map should be thought of as mapping a general positive definite (PD) matrix, with support $G^\circ$, to its corresponding \normprc matrix (ignoring the diagonal of all ones). We claim that the push-forward of $\Lc^{g+d}$ by $\xi$ is absolutely continuous w.r.t. to $\Hc^g = \Lc^g$ on $\reals^{G} \simeq \reals^{g}$; see Lemma~\ref{lem:rescaling:map:null:sets} in Section~\ref{sec:aux:lemmas} for a proof. Combined with the result that $\Hc^{g}( F (\Nc_\infty^G)) = 0$ of Step~1, this implies $\Lc^{g+d}(\xi^{-1}(N)) = 0$. But $\xi^{-1}(N)$ is the set of all PD matrices with support $G^\circ$ that are not perfect. The proof is complete.

\subsection{Proof of Theorem~\ref{thm:zero:measure:main}}\label{sec:proof:zero:measure}

Let us introduce some notation, most of which is borrowed from~\cite{Lnenicka2007} with minor modifications. Recall also our subsetting and indexing notations from Section~\ref{sec:Gauss:graph:mod}.

An \emph{$i j$-path} on $[d]$, of length $t+1$, is an ordered sequence $i_0 \to i_1 \to i_2 \dots \to i_t \to i_{t+1}$, where $i_j, j =0,\dots,t+1$ are distinct elements of $[d]$, $i_0= i$ and $i_{t+1} = j$. 
We  represent such a path as an ordered subset $\Pi = \{i_0,i_1\dots,i_{t+1}\}$ of $[d]$. An \emph{$i_0$-cycle} on $[d]$, of length $t+1$, is an $i_0 i_0$-path; that is, an ordered sequence of the form $i_0 \to i_1 \to i_2 \dots \to i_t \to i_0$, where $i_j, j =0,\dots,t$ are distinct elements of $[d]$.  We will represent such a cycle as an ordered subset $C = \{i_0,i_1\dots,i_t\}$ of $[d]$.  
Ultimately, the $ij$-paths and $i_0$-cycles will be used to represent non-intersecting paths and cycles on a graph $G$ on nodes $[d]$.

From here on, we consider the edges of a graph $G$ to be directed, i.e., ordered pairs of nodes. An undirected edge $ij\in G$ is interpreted as bidirected, i.e. $\{i,j\}\in G$ and $\{j,i\}\in G$.
We say that an $ij$-path $\Pi = \{i=i_0,i_1\dots,i_{t+1}=j\}$ belongs to $G$, denoted as  $\Pi \in G$, if all the edges in the path belong to $G$, that is, $i_{j} i_{j+1} \in G$ for all $j=0,\dots,t$. The set of $ij$-paths that belong to $G$ is denoted as $\Pc^{ij}(G)$. With some abuse of notation, we let $\Pc^{ij} = \Pc^{ij}([d])$ denote the set of all $ij$-paths on $[d]$ in the complete graph.
The set of all $ij$-paths of $G$ of length $t+1$ is denoted as $\Pc_{t}^{ij}(G)$, that is
\begin{align*}
\Pc^{ij}_t(G) := \{\Pi \in \Pc^{ij}:\; \Pi \in G,\; |\Pi| = t+1 \}.
\end{align*}
We let $\Pc_t(G) := \bigcup_{i,j \in [d]} \Pc_t^{ij}(G)$, the set of all paths of length $t+1$ in $G$.
The parallel notations for $i$-cycles, namely 
\begin{align}
\Cc^i(G), \;\;\Cc^i = \Cc^i([d]),\;\; \Cc^i_t(G), \;\; \text{and} \; \;\Cc_t(G)
\end{align}
are defined similarly (by setting $i=j$ in the corresponding definitions for paths). Note that in our notation, an undirected edge $ij \in G$, with $i\neq j$, is considered a valid cycle $\{i,j\}$ of length $2$, since both   $i\to j$ and $j\to i$ are in $G$.

For an $ij$-path $\Pi = \{i_0=i,i_1,\dots,i_{t},j=i_{t+1}\} \in \Pc^{ij}$ and a matrix $B = (b_{i,j}) \in \reals^{d \times d}$, let 
\begin{align}\label{eq:b:Pi}
b_\Pi =  \prod_{j=0}^t b_{i_j,i_{j+1}}.
\end{align}
If, instead, $b \in \reals^G$ and $\Pi$ is a path that belongs to $G$, then $b_\Pi$ as above is still well-defined (i.e., we do not need $b$ to be defined outside $G$.) A similar notation, namely $b_C$, is well-defined when $C$ is an $i_0$-cycle. (In this case, $i_{t+1} = i_0$ in~\eqref{eq:b:Pi}.)

\smallskip
The proof of Theorem~\ref{thm:zero:measure:main} relies on the following key technical lemma, which is as an extension of Lemma~4 in~\cite{Lnenicka2007} and is proven in Section~\ref{sec:proof:cycle}. Here, we treat node $1$ specially, hence the emphasis on the collection of $1$-cycles (cycles which begin and end on node $1$) of a given length $t+1$, namely $\Cc^1_t(G)$. The special role given to $1$ becomes clear in the proof of Theorem~\ref{thm:zero:measure:dual} below, where in dealing with an $ij$-path of $G$, we identify the endpoints with node 1 of a new graph $H$, hence obtaining a $1$-cycle of $H$. 

In the sequel, $\complex[x]$ is the set of polynomials in the indeterminate variable $x$ with complex coefficients. For $p \in \complex[x]$, we say that $p=0$ in $\complex[x]$, or $p(x) = 0$ in $\complex[x]$, if $p$ is the zero polynomial (i.e., all its coefficients are zero). For a square matrix $B$, $|B|$ denotes its determinant. 

\begin{lem}\label{lem:cycle}
	Consider a (directed) graph $H$ on $[r]$ with no self-loops on any node except possibly node $1$.
Let $\delta_{i,j}$ for $i,j\in[r]$ be a collection of nonzero real numbers.
Define a matrix $B(x) = (b_{i,j}(x)) \in \reals^{r \times r}$ by 
\begin{align*}
	b_{i,i} = 1, \; \forall i > 1, \qquad \text{and}\qquad b_{i,j}(x) = \delta_{i,j} \,x\, 1\big\{\{i,j\} \in H\big\},\;\; \text{for $i\neq j$ and $i=j=1$},
	\end{align*}
	treating $b_{i,j}(x)$ as a polynomial in $\complex[x]$. The following two statements hold: 
	\begin{enumerate}[(a)]
		\item  $|B(x)| = 0$ in $\complex[x]$ if $\Cc_t^1(H) = \emptyset$ for all $t \ge 0$.
		\item Assume further that for any $0\leq t <r$, 
		\begin{align}\label{eq:cycle:conditionexact}
		\sum_{C \,\in\, \Cc_t^1(H)} \delta_C \neq 0\quad \text{whenever $\Cc_t^1(H)$ is nonempty}.
		\end{align}
		Then, $|B(x)| = 0$ in $\complex[x]$ implies $\Cc_t^1(H) = \emptyset$ for all $t \ge 0$.
	\end{enumerate}
	
\end{lem}

Note that in the case $t = 0$, $\Cc_t^{1}(H)$ is nonempty only if $H$ has a self-loop on node $1$. 
Following~\cite{Lnenicka2007}, let $\Nc = [d]$ and $\Rc(\Nc)$ be the set of all couples $(ij|K)$ such that $i$ and $j$ are distinct singletons of $\Nc$ and $K \subset \Nc\setminus ij$. Subsets of $\Rc(\Nc)$ are called relations. To simplify notation, unless otherwise stated, couples of the form $(ij|K)$ are always assumed to belong to $\Rc(\Nc)$. 

\begin{defn}
	The dual couple of $(ij|K)$ is $(ij|\Nc\setminus ijK)$. 
	For a relation $\Lc \subset \Rc(\Nc)$, the \emph{dual relation} $\Lc^\dual$ is defined as the relation containing all the dual couples of the elements of $\Lc$.
\end{defn}
For any matrix $A \in \reals^{d \times d}$, let \begin{align*}
\ipp{A} := \big\{(ij|K) :\; \; |A_{iK,jK}| = 0\big\}.
\end{align*}
By Lemma~1 in~\cite{Lnenicka2007}, for an invertible matrix $A$, we have $\ipp{A}^\dual = \ipp{A^{-1}}$.  For a simple undirected graph $G$ with vertex set $\Nc$, let
\begin{align*}
\ip{G} := \big\{(ij|K):\; \;\text{$K$ separates $i$ and $j$ in $G$} \big\}.
\end{align*}
Recalling the notation $\Pc_t^{ij}(G)$, let us  write
$\Pc^{ij}_t(G;K) := \{\Pi \in \Pc^{ij}_t(G) :\; \Pi \subset ijK\}$ to denote the set of $ij$-paths in $G$ of length $t+1$ that pass entirely through $K$. Also recall that for $\delta \in \reals^G$ and any path $\Pi \in \Pc_t(G)$ (of length {$t+1$) in $G$, the quantity $\delta_\Pi$ is well-defined using~\eqref{eq:b:Pi}.
We define:
\begin{align}\label{eq:strong:path:condition}
\Dc = \Dc_G:= \Big\{ \delta \in [-1,1]_*^G:\; \sum_{\Pi \, \in\, \Pc^{ij}_t(G)} \delta_\Pi \neq 0,\; \;
\text{for all nonempty\,} \Pc^{ij}_t(G), \, ij \in G,\; 0 \le t < d \Big\}. 
\end{align}
Note that by definition, $\delta_{i,j} = \delta_{j,i}$ for $\delta \in \reals^G$.

\begin{thm}\label{thm:zero:measure:dual}
	Let $G$ be a simple graph with vertex set $\Nc$. Then, for any $\delta \in \Dc$, there are finitely many $\eps \in \complex$ for which $ \ip{G}^\dual = \ipp{A^{G,\delta,\eps}}$ fails, where $A^{G,\delta,\eps}$ is defined in~\eqref{eq:A:G:delta:eps}.
\end{thm}
\begin{proof}
	Let $A^{G,\delta,x}$ be the matrix with elements in $\complex[x]$ obtained by replacing $\eps$ in $A^{G,\delta,\eps}$ by indeterminate $x$. Consider 
	\begin{align*}
	\ipp{A^{G,\delta,x}}_{\complex[x]} := \{(ij|K):\; |A_{iK,jK}^{G,\delta,x}| = 0\;\,\text{in}\;\,\complex[x]\}. 
	\end{align*}

	\underline{Step 1.} Fix $\delta \in \Dc$. We show that $\ip{G}^\dual = \ipp{A^{G,\delta,x}}_{\complex[x]} $. Consider the matrix $A_{iK,jK}^{G,\delta,x}$, and 
	assume that its first row and column correspond to the $i$th row and $j$th column of $A^{G,\delta,x}$, by swapping rows and columns if necessary, noting that such operations do not change the determinant $|A_{iK,jK}^{G,\delta,x}|$.
We identify $(i,j)$ with element $(1,1)$ in Lemma~\ref{lem:cycle}. Let $H$ be the subgraph of $G$ induced on nodes $ijK$, with nodes $i$ and $j$ identified together and renamed node $1$.
	The only edges in $H$ that can be directed are those incident on node 1: $\{1,k\}\in H$ iff $\{i,k\}\in G$, and $\{k,1\}\in H$ iff $\{k,j\}\in G$. All the undirected edges in $H_{KK}$ are considered bidirected. In other words, the support of $A_{iK,jK}$ is the adjacency matrix of $H$, which can be asymmetric and thus correspond to a directed graph.
A path from $i$ to $j$ in $G$ that lies entirely in $ijK$ corresponds to a cycle in $H$ starting at node $1$, that is, we can identify $\Pc_t^{ij}(G;K)$ with $\Cc_t^1(H)$. A possible edge between $i$ and $j$ in $G$ will be a self-loop on node $1$ in $H$, i.e., $\delta_{ij} x$ plays the role of $\delta_{11}x$ in Lemma~\ref{lem:cycle}. 
Since $\delta\in\Dc$, it follows from \eqref{eq:strong:path:condition}
   	that assumption~\eqref{eq:cycle:conditionexact} holds for $H$ whenever
$(ij|K) \in \Rc(\Nc)$. It follows from Lemma~\ref{lem:cycle} that $|A_{iK,jK}^{G,\delta,x}| = 0$ if and only if $\Pc_t^{ij}(G;K)$ is empty for all $t \ge 0$. Hence, $i$ and $j$ are separated in $G$ by $\Nc\setminus ijK$, or in symbols $(ij|K) \in \ip{G}^\dual$, if and only if $|A_{iK,jK}^{G,\delta,x}| = 0$.
	
	\smallskip
	\underline{Step 2.} Fix $\delta \in \reals^G$. Then $|A_{iK,jK}^{G,\delta,x}| = 0\;\text{in}\;\complex[x]$ implies $|A_{iK,jK}^{G,\delta,\eps}| = 0$ for all $\eps \in \complex$.
That is, 
	\begin{align}\label{eq:strict:inclusion}
	\ipp{A^{G,\delta,x}}_{\complex[x]} \subset \ipp{A^{G,\delta,\eps}}, \quad \forall \eps \in \complex.
	\end{align}
The inclusion is strict if and only if there is $(ij|K)$ such that $p_{ijK}(x,\delta):= |A_{iK,jK}^{G,\delta,x}|$ is a nonzero polynomial  (in $\complex[x]$) with root $\eps$. Since any such polynomial has a finite number of roots, we have $  \ipp{A^{G,\delta,x}}_{\complex[x]} = \ipp{A^{G,\delta,\eps}} $, for all but finitely many $\eps \in \complex$. Combined with Step~1, the  assertion follows.
\end{proof}

\begin{rem}
The above proof contains the key intuition for defining~$\Dc$ as in~\eqref{eq:strong:path:condition}: For any $\delta\in \Dc_G$ and all but a finite number of $\eps$, $ \ip{G}^\dual = \ipp{A^{G,\delta,\eps}}$ and thus, $ \ip{G} = \ipp{A^{G,\delta,\eps}}^\dual=\ipp{\Sigma}$, where $\Sigma=(A^{G,\delta,\eps})^{-1}$ is the covariance matrix of $X$. This implies that
	$i-K-j$ in $G$ if and only if $|\Sigma_{iK,jK}|=0$, which is equivalent to $\rv_{i}\independent \rv_{j}\given \rv_K$ by Lemma~\ref{lem:covariance:char:pair} below. See the proof of Lemma~\ref{lem:Markov:Siginv:dual} for a rigorous argument.
\end{rem}

The following lemma is straightforward (see for example \cite{lattice-paper}):
\begin{lem}\label{lem:covariance:char:pair}
	Suppose $X \sim N(0,\Sigma)$ and $\Sigma \succ 0$. Then,
	$|\Sigma_{Si,Sj}| = 0$ is equivalent to $\rv_{i}\independent \rv_{j}\given \rv_S$
for all $i,j$ and $S \subset [d]_{ij}$.
\end{lem}

\begin{lem}\label{lem:Markov:Siginv:dual}
If $\ip{G}^\dual = \ipp{\Sigma^{-1}}$, then $\Sigma$ is Markov perfect w.r.t. $G$.
\end{lem}
\begin{proof}First we note that by Lemma~1 in~\cite{Lnenicka2007}, $\ip{G} = \ip{G}^{\dual \dual} = \ipp{\Sigma^{-1}}^\dual = \ipp{\Sigma}$. Recall $\Nc := [d]$. Assume that $\Sigma$ is not perfect. Then, there exist nonempty disjoint sets $A,B \subset \Nc$ and $K \subset \Nc \setminus AB$ such that $\rv_A \independent \rv_B \given \rv_K$,
	and $K$  does not separate $A$ and $B$. Then, $\exists i \in A, j \in B$ such that $\neg(\sep(i,K,j))$ and clearly $K \subset \Nc \setminus ij$ (i.e., $(ij|K) \in \Rc(\Nc)$). We also have 
	$\rv_{i}\independent \rv_{j}\given \rv_K$,
hence $|\Sigma_{Ki,Kj}| = 0$ by Lemma~\ref{lem:covariance:char:pair}. That is, $(ij|K) \in \ipp{\Sigma}$, hence we should have $(ij|K) \in \ip{G}$, contradicting $\neg(\sep(i,K,j))$. The proof is complete.
\end{proof}

\begin{proof}[Proof of Theorem~\ref{thm:zero:measure:main}]
	Part~(c) of Theorem~\ref{thm:zero:measure:main}, with $\Dc$ given by~\eqref{eq:strong:path:condition}, follows from Theorem~\ref{thm:zero:measure:dual} and Lemma~\ref{lem:Markov:Siginv:dual} and the relation $\Sigma^{-1} = A^{G,\delta,\eps}$. For part~(a), we note that $\Dc^c := \{\delta \in [-1,1]_*^G: \delta \notin\Dc\}$ is the finite union of the zero sets of nontrivial polynomials, hence of $\Lc^g$-measure zero in $[-1,1]_*^G$ (as a subset of $\reals^G \simeq \reals^g$). For part~(b), let $\sphere^G_\infty = \bigcup_{ij} (F_{ij}^+ \cup F_{ij}^-)$ be the decomposition of $\sphere^G_\infty$ into its $(g-1)$-dimensional faces: $F_{ij}^{\pm} = \{\delta:\; \delta_{ij} = \pm 1\}$. It is enough to show, for example, that $F_{ij}^{+} \cap \Dc^c$ has $\Hc^{g-1}$-measure zero. Let $G'$ be $G$ with edge $ij$ removed. By fixing $\delta_{ij} = 1$, we can view $F_{ij}^+ \cap \Dc^c$ as a subset of $F_{ij}^+ \subset \reals^{G'} \simeq \reals^{g-1}$.  Recalling the definition of $\Dc$,~\eqref{eq:strong:path:condition}, we observe, as before, that $F_{ij}^+ \cap \Dc^c$ as a subset of $\reals^{g-1}$ has $\Lc^{g-1}$-measure zero as a finite union of the zero sets of nontrivial polynomials in $g-1$ variables $\delta_{G'} = (\delta_{rs},rs \in G')$. Since $\Lc^{g-1} = \Hc^{g-1}$ on $\reals^{g-1}$, the assertion follows.
	
	For part~(d), both $\Lc^{g+1}(\Nc^G) = 0$ and $\Hc^{g}(\Nc_\infty^G) = 0$ follow from the Fubini theorem for the Lebesgue measure. For example, consider the latter assertion. It is enough to show $\Hc^{g}(\Nc^G \cap (F_{ij}^+ \times \reals)) = 0$. Viewing $\Nc^G \cap (F_{ij}^+ \times \reals)$ as a subset of $\reals^{g-1} \times \reals$, as above, and using the decomposition of the Lebesgue measure $\Lc^g = \Lc^{g-1} \times \Lc^1$, Fubini theorem gives
	\begin{align*}
	\Hc^{g}\big(\Nc^G \cap (F_{ij}^+ \times \reals)\big)
	&= \int_{F_{ij}^+} \Lc^1(B_\delta)  \, d\Hc^{g-1} (\delta)\\
	&= \int_{F_{ij}^+ \,\cap\, \Dc^c} \Lc^1(B_\delta)  \, d\Hc^{g-1} (\delta) + 
	\int_{F_{ij}^+ \,\cap\, \Dc} \Lc^1(B_\delta)  \, d\Hc^{g-1} (\delta).
	\end{align*}        
Both integrals are zero, the first since $\Hc^{g-1}(F_{ij}^+\cap \Dc^c) = 0$ by part~(b), and the second since $B_\delta$ has finitely many elements hence $\Lc^{1}(B_\delta) = 0$, by part~(c). The proof is complete. \end{proof}

\section{Proofs of auxiliary results}\label{sec:proof:aux}

We recall the following notational conventions: For a matrix $\Sigma \in \reals^{d \times d}$, and subsets $A,B \subset [d]$, we use $\Sigma_{A,B}$ for the submatrix on rows and columns indexed by $A$ and $B$, respectively. Single index notation is used for principal submatrices, so that $\Sigma_{A} = \Sigma_{A,\,A}$. For example, $\Sigma_{i,j}$ is the $(i,j)$th element of $\Sigma$ (using the singleton notation), whereas $\Sigma_{ij} = \Sigma_{ij,\,ij}$ is the $2\times 2$ submatrix on $\{i,j\}$ and $\{i,j\}$.

\subsection{Proof of Lemma~\ref{lem:cycle}}\label{sec:proof:cycle}
Recall the definition of the $i_0$-cycle (of $[d]$ or a graph $G$) from Section~\ref{sec:proof:zero:measure}. In proving  Lemma~\ref{lem:cycle}, we will  use the term \emph{cycle} to also refer to cycles of a permutation. 
The necessary background on cycle decomposition is briefly reviewed below. 
The two notions of cycle (graph versus permutation) are related in our arguments, and the distinction in each occurrence should be clear from the context.

Recall that every permutation $\pi$ on $[d]$, that is, a bijective map $\pi :[d] \to [d]$,  has a unique \emph{cycle decomposition}, once we agree on a particular order within cycles and among them~\cite[Section~1.3]{stanley1997enumerative}. For example, representing $\pi = (142)(35)$ means that $\pi$ has two cycles $C_1= \{1,4,2\}$ and $C_2=\{3,5\}$. $C_1$ being a cycle means that $\pi$ maps 1 to 4, 4 to 2, and 2 back to 1, and similarly for $C_2$. We treat the cycles of $\pi$ as ordered sets
with the smallest element written first, and the rest of the order determined by the action of $\pi$. (That is, if $C = \{i_0,i_1,\dots,i_t\}$ is a cycle of $\pi$, we have $i_0 < i_j$ and $\pi(i_{j-1}) = i_j$ for $j=1,\dots,t$.) Thus, permutation cycles are also graph cycles in the sense of Section~\ref{sec:proof:zero:measure}.  The (unordered) collection of cycles of $\pi$ will be denoted as $\Cyc_\pi$. In the example, $\Cyc_\pi =\{C_1,C_2\}$. The ordering among the cycles is unimportant.
In forming $\Cyc_\pi$, we disregard trivial cycles, that is, those containing a single element, except for the cycle containing $1$. We often talk about ``single cycle'' permutations: for example, $\pi' = (142)(3)(5)$ has a single cycle $C_1 = \{1,4,2\}$ in our convention, while $\pi'' = (1) (42)(3)(5)$ has two cycles $C_1 = \{1\}$ and $C_2 = \{42\}$. Similarly, the identity permutation has a single cycle in our convention. 

For matrix $B = (b_{i,j}) \in \reals^{d \times d}$ and permutation $\pi$ on $[d]$, we write
\begin{align}\label{eq:bpi:expan}
b_{\pi} := \prod_{i \in [d]} b_{i,\pi(i)} = \prod_{C \,\in\, \Cyc_\pi} b_{C},
\end{align}
where $b_C$ is as defined\footnote{The notation $b_\pi$ is also consistent with the definition of $b_C$ in  Section~\ref{sec:proof:zero:measure} due to the following connection: Every (graph) cycle $C$ can be viewed as a permutation that leaves elements outside $C$ intact.} in~\eqref{eq:b:Pi}. Since $b_{\{i\}} = b_{ii} = 1$ for $i\neq 1$, dropping single cycles $\{i\}$, for $i\neq 1$, from $\Cyc_\pi$ does not affect~\eqref{eq:bpi:expan}.  For the example above, the two expressions are 
\begin{align*}
b_\pi = b_{1,4} b_{2,1} b_{3,5} b_{4,2} b_{5,3} = (b_{1,4} b_{4,2} b_{2,1}) (b_{3,5} b_{5,3}).
\end{align*}
For any permutation $\pi$, let $C_\pi$ be its $1$-cycle, i.e., its cycle that contains $1$ and let $t_\pi = |C_\pi \setminus\{1\}| = |C_\pi| - 1$.  Note that $b_{C_\pi} = \prod_{i \in C_\pi} b_{i,\pi(i)}$ is a factor of $b_\pi$.

\begin{proof}[Proof of Lemma~\ref{lem:cycle}]
For simplicity, we will drop the explicit dependence on $x$ and write $B = (b_{i,j})$. It is well-known that
	\begin{align*}
	|B| = \sum_\pi \sign(\pi) b_\pi.
	\end{align*}
	First, consider the part~(a). Assume $\Cc_t^1(H) = \emptyset$ for all $t \ge 0$. The case $t = 0$ gives $\{1,1\}\notin H$, hence $b_{i,\pi(i)} = b_{1,1} = 0$ whenever $C_\pi = \{1\}$. Similarly, for any $C_\pi$ with $|C_\pi| > 1$, there are $i,j \in C_\pi$ with $i \neq j = \pi(i)$, such that $\{i,j\} \notin H$, hence $b_{i,\pi(i)} = 0$. Thus, $b_\pi = 0 $ for all $\pi$, giving $|B| = 0$ and proving part~(a).

	Now assume $|B| = 0$. 
We start by showing that $b_{C_\pi} = 0$ for all $\pi$.
We proceed by induction on $t_\pi = |C_\pi|-1$. Fix $0 \le t < r$. It suffices to show that if $b_{C_\pi} = 0$ for all $\pi$ with $t_\pi < t$, then $b_{C_\pi} = 0$ for all $\pi$ with $t_\pi = t$. The same argument below, with $t=0$, establishes the initial step of the induction. For any cycle $C$, \begin{align}\label{eq:bC:expr}
	b_C = \delta_C x^{|C|} 1\{C \in H\},
	\end{align}
	that is, $b_{C}$ is equal to 0 or $\delta_{C} x^{|C|}$, the latter if and only if $C \in H$. Here, $\delta_C$ is defined similar to $b_C$.  

	 By the induction assumption, it follows that $b_\pi = 0 $ for all $\pi$ for which $t_\pi < t$ since $b_{C_\pi}$ is a factor of $b_\pi$.
		 It follows that  $0 = |B| = \sum_{\pi :\; t_\pi \,\ge\, t} \sign(\pi) b_\pi$. There are three types of terms in this expansion: (Below, $\Cyc_\pi$ is the cycle decomposition of $\pi$, using the convention discussed earlier.)
	\begin{enumerate}[(a)] 
		\item $|\Cyc_\pi| = 1, t_\pi = t$: 
These have a cycle $C_\pi$ of length $t+1$ containing 1, and every other cycle is trivial.
		All of these permutations have the same sign, and we have
		\begin{align}\label{eq:temp:4875}
		b_\pi  = b_{C_\pi} = \delta_{C_\pi} x^{t+1} 1\{C_\pi \in H\}.
		\end{align}
		The first equality is since $b_{i,i} = 1$ for all $i \neq 1$.  
As $\pi$ varies over the permutations in this category, $C_\pi$ runs over all $\Cc_t^{1}$, i.e., cycles of length $t+1$ over $[r]$ containing~$1$. That is,
		\begin{align*}
		 \{C_\pi : t_\pi = t \} = \Cc_t^{1}. 
		\end{align*}
		(Note that the correspondence also holds for $t = 1$ since the edges as considered directed. E.g., the permutation cycle $C_\pi = (12)$ corresponds to the graph cycle $1 \to 2$ and $2 \to 1$ in $\Cc_1^1$. In this case, we have $b_\pi = \delta_{12}\delta_{21} x^2 \big\{\{1,2\} \in H, \{2,1\} \in H \big\}$.)

		However, only the subset $\Cc_t^{1}(H)$ of $\Cc_t^{1}$ contributes to $|B|$ due to the indicator $1\{C_\pi \in H\}$ in~\eqref{eq:temp:4875}.
		There are two possible cases:
		\begin{enumerate}[(i)]
			\item $\Cc_t^{1}(H) = \emptyset$; then $b_{C_\pi} = 0$ for all $\pi$ such that $|\Cyc_\pi| = 1$ and $t_\pi = t$.
			
			\item $\Cc_t^{1}(H) \neq \emptyset$; then,   these permutations contribute to $|B|$, a term $\pm \big(\sum_{C\,\in\, \Cc_t^{1}(H)} \delta_C \big) x^{t+1}$.
		\end{enumerate}

\item $|\Cyc_\pi| \ge 2, t_\pi = t$: Any such permutation has at least a cycle $C$ of size $\nu \ge 2$ in $[r] \setminus C_\pi$. Hence, $b_\pi$ has a factor of the form
		\begin{align*}
		b_{C_\pi} b_C = \delta_{C_\pi} \delta_{C}\, x^{t+\nu+1} 1\{C_\pi,C \in H\}
		\end{align*}
		Thus, any such $b_\pi$, if nonzero, contributes a polynomial of degree at least $t+3$.
		
		\item $t_\pi \ge t+1$: In this case, $b_\pi$ has a factor of $b_{C_\pi}  = \delta_{C_\pi} x^{t_\pi+1} 1\{C_\pi \in H\}$
		and as the previous case contributes a polynomial of degree at least $t+2$, if nonzero. \end{enumerate}
	
	Thus, the coefficient of $x^{t+1}$ in $|B|$ is determined only by permutations of type~(a). But, since this coefficient is zero by the assumption that $|B|=0$, we conclude that case (ii)  above cannot occur, since then $\sum_{C \in \Cc_t^{1}(H)} \delta_C = 0$ for some nonempty $\Cc_t^{1}(H)$ with  $t \in \{0,\dots,r-1\}$, contradicting assumption~\eqref{eq:cycle:conditionexact}.

	This in turn implies that
	for any permutation $\pi$ of type~(a), we have $b_{C_\pi} = 0$, by~\eqref{eq:temp:4875} and that $H$ cannot contain any cycle of size $t+1$. But this proves the induction claim: For any permutation $\pi'$ with $t_{\pi'} = t$, there is permutation $\pi$ of type~(a) such that $C_\pi = C_{\pi'}$ (i.e. break all the cycles of $\pi'$, other than $C_{\pi'}$, into trivial ones).
	
	As a byproduct of establishing the induction claim, we also obtain $\Cc_t^1(H) = \emptyset$ for all $t \ge 0$ which is the desired result. (In particular, with $t=0$, it means that $H$ cannot have a self-loop on node $1$ if $|B| = 0$.)  The proof is complete.
\end{proof}

\subsection{Auxiliary lemmas}\label{sec:aux:lemmas}

The following lemma is used in the proof of Theorem~\ref{thm:perfect:perval}. The notation $\xi_* \mu$ denotes the push-forward of measure $\mu$ by map $\xi$.

\begin{lem}\label{lem:rescaling:map:null:sets}
	With $\xi: \reals_{++}^d \times \reals^g \to \reals^g$ defined as in~\eqref{eq:rescaling:map:defn}, we have $\xi_* \Lc^{d+g} \ll \Lc^g$, that is, $\Lc^{g}(A) = 0$ implies $\Lc^{g+d}(\xi^{-1}(A)) = 0$. 
\end{lem}

\begin{proof}
	Let $\Omega := \reals_{++}^d \times \reals^g$ be a subset of $\reals^{d+g}$. Let $x = (x_k, \;k \in [d])$ and $y = (y_{ij},\; ij \in G)$. Consider the function $F_1: \Omega \to \Omega$ defined by 
	\begin{align*}
	F_1(x,y) = \Big(x, \frac{y_{ij}}{\sqrt{x_i x_j}},\, ij \in G\Big).
	\end{align*}
	$F_1$ is a $C^\infty$ diffiomorphism of $\Omega$ onto itself, that is, $F_1: \Omega \to \Omega$ is a bijection and both $F_1$ and its inverse $F_2:= F_1^{-1}$ belong to class $C^\infty$. This implies that $F_1$ and $F_2$ are locally Lipschitz (i.e., Lipschitz when restricted to any compact subset of $\Omega$), hence they both preserve $\Lc^{g+d}$-null sets (i.e., map null sets to null sets). 
	
	Let $\pi : \reals^{d+g} \to \reals^g$ be the projection $\pi(x,y) = y$. We can write $\xi = \pi \circ F_1$. 
We  first show that $\pi_* \Lc^{d+g} \ll \Lc^g$. This follows from Fubini theorem: Let $A \subset \reals^g$ be such that $\Lc^g(A) = 0$. We have $\pi^{-1}(A) = \reals^d \times A$. Hence, $\Lc^{d+g}(\pi^{-1}(A)) = \Lc^d(\reals^d) \cdot \Lc^g(A) = 0$ since the Lebesgue measure is $\sigma$-finite.
	
	Now assuming that $\Lc^g(A) = 0$, we thus have $\Lc^{g+d}(\pi^{-1}(A)) = 0$. But then $\Lc^{g+d}(F_2 \circ \pi^{-1} (A)) = 0 $, due to the diffiomorphic nature of $F_2$. Noting that $\xi^{-1} = (\pi \circ F_1)^{-1} = F_1^{-1} \circ \pi^{-1} = F_2 \circ \pi^{-1}$, we have the desired result. The proof is complete.
\end{proof}

\section*{Acknowledgement}

This work was supported in part by NSF grant IIS-1546098.

\printbibliography

\end{document}